\documentclass[12pt, a4paper, reqno]{amsart}
\usepackage{amssymb,amsmath,amsthm}
\usepackage{a4wide}
\usepackage{color}
\usepackage{enumerate}
\usepackage{url}
\usepackage[mathscr]{eucal}
\usepackage{setspace}
\newcommand{\sm}[4]{\left(\begin{smatrix}#1&#2\\ #3&#4 \end{smallmatrix} \right)}
\newtheorem{theorem}{Theorem}
\newtheorem{lemma}[theorem]{Lemma}
\newtheorem{corollary}[theorem]{Corollary}

\newtheorem{proposition}[theorem]{Proposition}
\newtheorem{definition}[theorem]{Definition}

\theoremstyle{remark}

\newtheorem*{remark}{Remark}

\numberwithin{theorem}{section} \numberwithin{equation}{section}

\newcommand{\Mp}{\text {\rm Mp}}

\newcommand{\R}{\mathbb{R}}
\newcommand{\C}{\mathbb{C}}

\newcommand{\Q}{\mathbb{Q}}

\newcommand{\Z}{\mathbb{Z}}

\newcommand{\SL}{{\text {\rm SL}}}

\newcommand{\sgn}{\operatorname{sgn}}
\newcommand{\PSL}{{\text {\rm PSL}}}
\DeclareMathOperator{\dlconst}{\mathfrak{d}}
\DeclareMathOperator{\Iso}{Iso}
\DeclareMathOperator{\SO}{SO}

\newcommand{\h}{\mathbb{H}}
\newcommand{\G}{\Gamma}
\newcommand{\dg}{\mathcal{D}} 
\newcommand{\dgdelta}{{\mathcal{D}(\Delta)}} 

\newcommand{\abs}[1]{\left\vert#1\right\vert}
\newcommand{\e}{\mathfrak{e}}

\newcommand{\smallabcd}{\left(\begin{smallmatrix}a & b \\ c & d\end{smallmatrix}\right)}
\newcommand{\Deltaover}[1]{\left(\frac{\Delta}{#1}\right)}
\newcommand{\kronecker}[2]{\left(\frac{#1}{#2}\right)}

\newcommand{\smallTmatrix}{\left(\begin{smallmatrix}1 & 1 \\ 0 & 1\end{smallmatrix}\right)}
\newcommand{\smallSmatrix}{\left(\begin{smallmatrix}0 & -1 \\ 1 & 0\end{smallmatrix}\right)}

\newcommand{\mt}{\mathbf{t}}

\begin{document}
\singlespacing
\title[]{Twisted traces of CM values of weak Maass forms}

\author{Claudia Alfes}
\email[C.~Alfes]{alfes@mathematik.tu-darmstadt.de}
\author{Stephan Ehlen}
\email[S.~Ehlen]{ehlen@mathematik.tu-darmstadt.de}
\address{Fachbereich Mathematik, Technische Universit\"{a}t Darmstadt, Schlossgartenstra\ss{}e 7, D--64289 Darmstadt, Germany}


\date{\today}
\begin{abstract}
We show that the twisted traces of CM values of weak Maass forms of weight 0 
are Fourier coefficients of vector valued weak Maass forms of 
weight 3/2. These results generalize work by Zagier on traces of singular moduli.
We utilize a twisted version of the theta lift 
considered by Bruinier and Funke \cite{BrFu06}.
\end{abstract}
\maketitle

\section{Introduction}

The values of the modular invariant $j(z)$ at quadratic irrationalities, classically called ``singular moduli'', are known to be algebraic integers. Their properties have been intensively studied since the 19th century. In an influential paper \cite{Zagier}, Zagier showed that the (twisted) traces of singular moduli are Fourier coefficients of weakly holomorphic modular forms of weight $3/2$. Recall that these are meromorphic modular forms which are holomorphic on the complex upper half plane 
$\h = \{ z \in \C;\ \Im(z) > 0\}$ with possible poles at the cusps.

Throughout this paper, we let $N$ be a positive integer and we denote by
$\G$ the congruence subgroup 
$\G=\Gamma_0(N)=\left\lbrace\smallabcd \in \SL_2(\Z); c \equiv 0 \bmod N \right\rbrace$.
For a negative integer $D$ congruent to a square modulo $4N$, we 
consider the set $\mathcal{Q}_{D,N}$ 
of \textit{positive definite} integral binary quadratic forms 
$\left[a,b,c\right]=ax^2+bxy+cy^2$ of discriminant $D=b^2-4ac$ 
such that $c$ is congruent to $0$ modulo $N$.
If $N=1$, we simply write $\mathcal{Q}_{D}$.
For each form $Q = \left[a,b,c\right] \in \mathcal{Q}_{D,N}$
there is an associated CM point $\alpha_Q=\frac{-b+i\sqrt{D}}{2a}$ 
in $\h$. The group $\G$ acts on $\mathcal{Q}_{D,N}$ with finitely many orbits.

Let $\Delta \in \Z$ be a fundamental discriminant (possibly 1) and $d$ 
a positive integer such that $-\sgn(\Delta)d$ and $\Delta$ are squares modulo $4N$. 
For a weakly holomorphic modular form $f$ of weight 0 for $\G$, 
we consider the modular trace function
\begin{equation}
 \mt_\Delta(f;d)=\frac{1}{\sqrt{\Delta}}\sum\limits_{Q\in\G\backslash\mathcal{Q}_{-d\abs{\Delta},N}}\frac{\chi_{\Delta}(Q)}{\abs{\overline\G_Q}}f(\alpha_Q).
\end{equation}
Here $\overline{\G}_Q$ denotes the stabilizer of $Q$ in $\overline{\G}$, the image of $\G$ in $\PSL_2(\Z)$.
The function $\chi_\Delta$ is a genus character, defined for 
$Q=[a,b,c] \in \mathcal{Q}_{-d\abs{\Delta},N}$ by
\begin{equation}\label{intro:chi}
 \chi_\Delta(Q)= \begin{cases}
				\kronecker{\Delta}{n}, &\text{ if } (a,b,c,\Delta)=1 \text{ and } Q \text{ represents } n \text{ with } (n,\Delta)=1,\\
				0, &\text{otherwise}.
			   \end{cases}
\end{equation}
It is known that $\chi_\Delta(Q)$ is $\G$-invariant \cite{GKZ}.
Note that for $\Delta=1$ we have $\chi_\Delta(Q)=1$ for all $Q \in \mathcal{Q}_{-d,N}$.

Let $J(z)=j(z)-744=q^{-1}+196884q+21493760q^2+\cdots$, $q:=e^{2\pi i z}$, be the normalized Hauptmodul for the group $\mathrm{PSL}_2(\Z)$. 
By the theory of complex multiplication it is known that $\mt_\Delta(J;d)$
is a rational integer \cite[Section 5.4]{ShimAuto}.

Zagier \cite[Theorem 6]{Zagier} proved that for $N=1$ and $\Delta > 0$ the ``generating series'' of these traces,
\begin{equation}\label{eq:intro1}
g_\Delta(\tau) = q^{-\Delta} - \sum\limits_{d\geq 0} \mt_\Delta(J;d) q^d,
\end{equation}
is a weakly holomorphic modular form of weight $3/2$ for $\G_0(4)$. Here, we set $\mt_\Delta(J;0)=2$, if $\Delta=1$ and $\mt_\Delta(J;0)=0$, otherwise.

Using Hecke operators, 
it is also possible to obtain from this
a formula for the traces of $J_m$, 
the unique weakly holomorphic modular function for $\SL_2(\Z)$ 
with principal part equal to $q^{-m}$, where $m$ is a positive integer.
Namely, we have
\begin{equation}\label{eq:intro2}
	\mt_\Delta(J_m;d) = \sum_{n \mid m} \kronecker{\Delta}{m/n} n\ \mt_\Delta(J;d n^2).
\end{equation}
Many authors \cite{BringmannOno, DukeJenkins, KimTwisted, MillerPixton}
worked on generalizations of these results, 
mostly also for modular curves of genus 0. 
Inspired by previous work \cite{Funke}, Bruinier and Funke \cite{BrFu06} 
showed that the function $g_1(\tau)$ can be interpreted
as a special case of a theta lift using a kernel function
constructed by Kudla and Millson \cite{KM86}.
They generalized Zagier's result for $\Delta=1$ to traces of 
harmonic weak Maass forms of weight $0$
on modular curves of \textit{arbitrary} genus \cite[Theorem 7.8]{BrFu06}.

The purpose of the present paper is to extend these results to the case $\Delta \neq 1$. 
We develop a systematic approach to twist vector valued modular forms
transforming with a certain Weil representation of $\Mp_2(\Z)$. 
From a representation--theoretic point of view, this can be interpreted
as an intertwining operator between two related representations. 
Then we define a twisted version of the theta lift studied by 
Bruinier and Funke and apply their results to obtain its Fourier expansion.

To illustrate our main result, Theorem \ref{thm:main}, we consider the following special case.
\begin{theorem}\label{thm:intro}
  Let $p$ be a prime or $p=1$ and let $f$ be a weakly holomorphic modular form of weight $0$ for $\G_0(p)$. Assume that $f$ is invariant under the Fricke involution $z \mapsto -\frac{1}{pz}$ and write $f(z)=\sum_{n\gg-\infty}a(n)q^n$ for its Fourier expansion.
 Let $\Delta > 1$ and, if $p \neq 1$, assume that $(\Delta,2p)=1$. Then the function
 \[
  \sum\limits_{m>0}m\sum\limits_{n>0}\left(\frac{\Delta}{n}\right) a(-mn)q^{-\abs{\Delta}m^2} - \sum\limits_{d>0}\mt_\Delta^*(f;d)q^d
 \]
is a weakly holomorphic modular form of weight $3/2$ for $\G_0(4p)$ contained in the Kohnen plus-space.
Here, $\mt_\Delta^*(f;d)=\mt_\Delta(f;d)/2$, for $p \neq 1$, and $\mt_\Delta^*(f;d)=\mt_\Delta(f;d)$,
for $p=1$. 
\end{theorem}

\begin{remark}
 Note that in contrast to the case $\Delta=1$ \cite[Theorem 1.1]{BrFu06}
 we do not get a constant term here and the lift is always weakly holomorphic.
 Setting $f=J_m$ and $p=1$, we also obtain the forms $g_\Delta(\tau)$ and
 the relation \eqref{eq:intro2} from Theorem \ref{thm:intro}.
 
 Furthermore, note that our main theorem is also valid for $\Delta<0$.
\end{remark}

The paper is organized as follows. In Section 2 we review necessary background material on quadratic spaces and vector valued automorphic forms. In Section 3 we show that twisting can be regarded as an intertwining operation. Then we define a twisted Kudla-Millson theta kernel and in Section 5 we prove our main theorem. In Section 6 we compute the twisted theta lift for other types of automorphic forms following the examples studied by Bruinier and Funke. 
In particular, we show that a twisted intersection pairing \`a la Kudla, Rapoport, and Yang \cite{KRY} 
is given in terms of a weight $3/2$ Eisenstein series. 
We also explain how to deduce Theorem \ref{thm:intro}, and present a few computational examples.

\section*{Acknowledgments}
We would like to thank Jan Bruinier for suggesting this project and for many
valuable discussions. We thank Jens Funke for substantially improving the
exposition of the proof of our main theorem. 
Moreover, we would like to thank Martin H\"ovel and
Fredrik Str\"omberg for helpful comments on earlier versions of this paper.

\section{Preliminaries}

For a positive integer $N$ we consider the rational quadratic space of signature $(1,2)$ given by
\[
V:=\left\{\lambda=\begin{pmatrix} \lambda_1 &\lambda_2\\\lambda_3& -\lambda_1\end{pmatrix}; \lambda_1,\lambda_2,\lambda_3 \in \Q\right\}
\]
and the quadratic form $Q(\lambda):=N\text{det}(\lambda)$.
The corresponding bilinear form is $(\lambda,\mu)=-N\text{tr}(\lambda \mu)$ for $\lambda, \mu \in V$.

Let $G=\mathrm{Spin}(V) \simeq \SL_2$, viewed as an algebraic group over $\Q$
and write $\overline\G$ for its image in $\mathrm{SO}(V)\simeq\mathrm{PSL}_2$.
Let $D$ be the associated symmetric space realized as the Grassmannian of lines 
in $V(\R)$ on which the quadratic form $Q$ is positive definite,
\[
D \simeq \left\{z\subset V(\R);\ \text{dim}z=1 \text{ and } Q\vert_{z} >0 \right\}.
\]
In this setting the group $\SL_2(\Q)$ acts on $V$ by conjugation
\[
  g.\lambda :=g \lambda g^{-1},
\]
for $\lambda \in V$ and $g\in\SL_2(\Q)$. In particular, $G(\Q)\simeq\SL_2(\Q)$.

If we identify the symmetric space $D$ with the complex upper half plane $\h$ in the usual way,
we obtain an isomorphism
between $\h$ and $D$ by
\[
 z \mapsto \R \lambda(z),
 \]
where, for $z=x+iy$, we pick as a generator for the associated positive line
 \[
 \lambda(z):=\frac{1}{\sqrt{N}y} \begin{pmatrix} -(z+\bar{z})/2 &z\bar{z} \\  -1 & (z+\bar{z})/2 \end{pmatrix}.
 \]
The group $G$ acts on $\h$ by linear fractional transformations and
the isomorphism above is $G$-equivariant.
In particular, $Q\left(\lambda(z)\right)=1$ and $g.\lambda(z)=\lambda(gz)$ for $g\in G(\R)$.
Let $(\lambda,\lambda)_z=(\lambda,\lambda(z))^2-(\lambda,\lambda)$. This is the minimal majorant of $(\cdot,\cdot)$ associated with $z\in D$.

We can view $\G=\G_0(N)$ as a discrete subgroup of $\mathrm{Spin}(V)$. 
Write $M=\G \setminus D$ for the attached locally symmetric space.

The set of isotropic lines $\mathrm{Iso}(V)$ in $V(\Q)$ can be identified
with $P^1(\Q)=\Q \cup \left\{ \infty\right\}$ via
\[
\psi: P^1(\Q) \rightarrow \mathrm{Iso}(V), \quad \psi((\alpha:\beta))
 = \mathrm{span}\left(\begin{pmatrix} \alpha\beta &\alpha^2 \\  -\beta^2 & -\alpha\beta \end{pmatrix}\right).
\]
The map $\psi$ is a bijection and $\psi(g(\alpha:\beta))=g.\psi((\alpha:\beta))$. 
So the cusps of $M$ (i.e. the $\G$-classes of $P^1(\Q)$) can be identified with the $\G$-classes of $\mathrm{Iso}(V)$.

If we set $\ell_\infty := \psi(\infty)$, then $\ell_\infty$ is spanned by 
$\lambda_\infty=\left(\begin{smallmatrix}0 & 1 \\ 0 & 0\end{smallmatrix}\right)$. 
For $\ell \in \mathrm{Iso}(V)$ we pick $\sigma_{\ell} \in\SL_2(\Z)$ 
such that $\sigma_{\ell}.\ell_\infty=\ell$. 
Furthermore, we orient all lines $\ell$ by requiring that 
$\lambda_\ell:=\sigma_{\ell}\lambda_\infty$ is a positively oriented 
basis vector of $\ell$.
Let $\Gamma_{\ell}$ be the stabilizer of the line $\ell$. 
Then (if $-I\in \Gamma$)
\[
\sigma_{\ell} ^{-1}\Gamma_{\ell} \sigma_{\ell} 
= \left\{\pm \begin{pmatrix} 1 &k\alpha_{\ell} \\   & 1 \end{pmatrix}; k\in\Z \right\},
\]
where $\alpha_{\ell} \in \Q_{>0}$ is the width of the cusp $\ell$ \cite{Funke}. 
In our case it does not depend on the choice of $\sigma_{\ell}$. 
For each $\ell$ there is a $\beta_{\ell} \in \Q_{>0}$ such that 
$\left(\begin{smallmatrix}0 &  \beta_{\ell} \\ 0 & 0\end{smallmatrix}\right)$ 
is a primitive element of $\ell_\infty \cap\sigma_{\ell}L$.
We write $\epsilon_{\ell}  = \alpha_{\ell} /\beta_{\ell}$.

Now Heegner points are given as follows.
For $\lambda\in V(\Q)$ with $Q(\lambda)>0$ let
\[
D_{\lambda}= \mathrm{span}(\lambda) \in D.
\]
For $Q(\lambda) \leq 0$ we set $D_{\lambda}=\emptyset$.
We denote be the image of $D_{\lambda}$ in $M$ by $Z(\lambda)$.

If $Q(\lambda)<0$, we obtain a geodesic $c_{\lambda}$ in $D$ via
\[
c_{\lambda}=\left\{z \in D; z \perp \lambda \right\}.
\]
We denote $\Gamma_{\lambda} \backslash c_{\lambda}$ in $M$ by $c(\lambda)$.
The stabilizer $\overline\G_\lambda$ is either trivial or infinite cyclic.
The geodesic $c(\lambda)$ is infinite if and only if the following equivalent conditions hold \cite{Funke}.
 \begin{enumerate}
	 \item We have $Q(\lambda)\ \in\ -N(\Q^{\times})^2$.
	 \item The stabilizer $\overline\G_\lambda$ is trivial.
	 \item The orthogonal complement $\lambda^{\perp}$ is split over $\Q$.
 \end{enumerate}
Thus if $c(\lambda)$ is an infinite geodesic, $\lambda$ is orthogonal 
to two isotropic lines $\ell_\lambda=\text{span}(\mu)$ 
and $\tilde{\ell}_\lambda=\text{span}(\tilde{\mu})$, with $\mu$ and $\tilde{\mu}$ positively oriented.
We fix an orientation of $V$ and we say that $\ell_\lambda$ is 
the line associated with $\lambda$ if the triple $(\lambda,\mu,\tilde{\mu})$ 
is a positively oriented basis for $V$. In this case, we write $\lambda \sim \ell_\lambda$.

\subsection{A lattice related to $\G_0(N)$}

Following Bruinier and Ono \cite{BrOno}, we consider the lattice
\[
 L:=\left\{ \begin{pmatrix} b& -a/N \\ c&-b \end{pmatrix}; \quad a,b,c\in\Z \right\}.
\]
The dual lattice corresponding to the bilinear form $(\cdot,\cdot)$ is given by
\[
 L':=\left\{ \begin{pmatrix} b/2N& -a/N \\ c&-b/2N \end{pmatrix}; \quad a,b,c\in\Z \right\}.
\]
We identify the discriminant group $L'/L=:\dg$ 
with $\Z/2N\Z$, together with the $\Q/\Z$ valued quadratic form ${x \mapsto -x^2/4N}$.
The level of $L$ is $4N$. We note that Bruinier and Ono \cite{BrOno} consider the same lattice 
together with the quadratic form $-Q$.

For a fundamental discriminant $\Delta\in\Z$ we also consider the rescaled lattice $\Delta L$ together with the quadratic form $Q_\Delta(\lambda):=\frac{Q(\lambda)}{\abs{\Delta}}$. The corresponding bilinear form is given by $(\cdot,\cdot)_\Delta = \frac{1}{\abs{\Delta}} (\cdot,\cdot)$. The dual lattice of $\Delta L$ corresponding to $(\cdot,\cdot)_\Delta$ is equal to $L'$ as above, independent of $\Delta$. We denote the discriminant group $L'/\Delta L$ by $\dgdelta$. Note that $\dg(1)=\dg$ and $\abs{\dgdelta}=\abs{\Delta}^3 \abs{\dg} = 2N\, \abs{\Delta}^3$.

Note that $\G_0(N) \subset \mathrm{Spin}(L)$ is a congruence subgroup of $\mathrm{Spin}(L)$ which takes $L$ to itself and acts trivially on the discriminant group $\dg$. However, in general it does not act trivially on $\dgdelta$.

For $m \in \Q$ and $h \in \dg$, we will also consider the set
\begin{equation}
 L_{h,m}  = \left\{ \lambda \in L+h; Q(\lambda)=m  \right\}.
\end{equation}
By reduction theory, if $m \neq 0$ the group $\Gamma$ acts on $ L_{h,m}$ with finitely many orbits.


\subsection{The Weil representation and vector valued automorphic forms}
We denote by $\Mp_2(\Z)$ the integral metaplectic group, which consists of pairs
$(\gamma, \phi)$, where $\gamma = {\smallabcd \in \SL_2(\Z)}$ and $\phi:\h\rightarrow \C$
is a holomorphic function with $\phi^2(\tau)=c\tau+d$.
The group $\Mp_2(\Z)$ is generated by $S=(\smallSmatrix,\sqrt{\tau})$ and $T=(\smallTmatrix, 1)$.
We consider the Weil representation $\rho_\Delta$ of $\Mp_2(\Z)$ 
corresponding to the discriminant group $\dgdelta$ on the group ring $\C[\dgdelta]$,
equipped with the standard scalar product $\langle \cdot , \cdot \rangle$, conjugate-linear
in the second variable. We simply write $\rho$ for $\rho_1$.

For $\delta \in \dgdelta$, we write $\e_\delta$ 
for the corresponding standard basis element of $\C[\dgdelta]$.
The action of $\rho_\Delta$ on basis vectors of $\C[\dgdelta]$ 
can be described in terms of the following formulas \cite{BrHabil} for 
the generators $S$ and $T$ of $\Mp_2(\Z)$. In our special case we have
\begin{equation*}
 \rho_\Delta(T) \e_\delta = e(Q_\Delta(\delta)) \e_\delta,
\end{equation*}
and
\begin{equation*}
 \rho_\Delta(S) \e_\delta = \frac{\sqrt{i}}{\sqrt{\abs{\dgdelta}}}
				\sum_{\delta' \in \dgdelta} e(-(\delta',\delta)_\Delta) \e_{\delta'}.
\end{equation*}
Let $k \in \frac{1}{2}\Z$, and let $A_{k,\rho_\Delta}$ be the vector space of functions 
$f: \h \rightarrow \C[\dgdelta]$, such that for $(\gamma,\phi) \in \Mp_2(\Z)$ we have 
\begin{equation}
	f(\gamma \tau) = \phi(\tau)^{2k} \rho_\Delta(\gamma, \phi) f(\tau).
\end{equation}
A twice continuously differentiable function $f\in A_{k,\rho_\Delta}$ 
is called a \textit{(harmonic) weak Maass form of weight $k$ with respect to the representation $\rho_\Delta$} 
if it satisfies in addition:
\begin{enumerate}[(i)]
 \item $\Delta_k f=0$,
 \item there is a $C>0$ such that $f(\tau)=O(e^{Cv}) $ as $v \rightarrow \infty$.
 \end{enumerate}
Here and throughout, we write $\tau=u+iv$ with $u,v \in \R$ and
$\Delta_k=-v^2\left(\frac{\partial^2}{\partial u^2}+\frac{\partial^2}{\partial v^2}\right)
+ikv\left(\frac{\partial}{\partial u}+i\frac{\partial}{\partial v}\right)$ 
is the weight $k$ Laplace operator.
We denote the space of such functions by $H_{k,\rho_\Delta}$. 
Moreover, we let $H^{+}_{k,\rho_\Delta}$ be the subspace of functions in $H_{k,\rho_\Delta}$ whose singularity at $\infty$ is locally given by the pole of a meromorphic function.
By $M^{\text{!}}_{k,\rho_\Delta} \subset H^{+}_{k,\rho_\Delta}$ 
we denote the subspace of weakly holomorphic modular forms.
 
Similarly, we can define scalar valued analogues of these spaces of automorphic forms.
In those cases, we require analogous conditions at all cusps of $\G$ in $(ii)$.
We denote the corresponding spaces by $H^{+}_{k}(\G)$ and $M_k^{\text{!}}(\G)$.

Note that the Fourier expansion of any harmonic weak Maass form uniquely decomposes into a holomorphic and a non-holomorphic part \cite[Section 3]{BrFu04}. For example, for $f\in H^{+}_0(\G)$ we have
\begin{equation}\label{eq:fourierhmf0}
f(\sigma_\ell  \tau) = \sum_{\substack{n\in \frac{1}{\alpha_{\ell}}\Z\\n\gg -\infty}} a^{+}_{\ell}(n)e(n \tau)
			  + \sum_{\substack{n\in \frac{1}{\alpha_{\ell}}\Z\\n<0}} a^{-}_{\ell}(n)e(n\bar{\tau}),
\end{equation}
where $\alpha_{\ell}$ denotes the width of the cusp $\ell$ and the first summand is called the holomorphic part of $f$, the second one the non-holomorphic part.

%
%

\section{Twisting vector valued modular forms}\label{sec:twisting}
We now define a generalized genus character for 
$\delta = \left(\begin{smallmatrix} b/2N& -a/N \\ c&-b/2N \end{smallmatrix}\right) \in L'$. 
Let $\Delta\in\Z$ be a fundamental discriminant and $r\in\Z$ such that $\Delta \equiv r^2 \ (\text{mod } 4N)$.

We let
\begin{equation}\label{def:chidelta}
\chi_{\Delta}(\delta)=\chi_{\Delta}(\left[a,b,Nc\right]):=
\begin{cases}
\Deltaover{n}, & \text{if } \Delta | b^2-4Nac \text{ and } (b^2-4Nac)/\Delta \text{ is a}
\\
& \text{square modulo } 4N \text{ and } \gcd(a,b,c,\Delta)=1,
\\
0, &\text{otherwise}.
\end{cases}
\end{equation}
Here, $\left[a,b,Nc\right]$ is the integral binary quadratic form 
corresponding to $\delta$, and $n$ is any integer prime to $\Delta$ represented by $\left[a,b,Nc\right]$.

The function $\chi_{\Delta}$ is invariant under the action of $\G_0(N)$ and under the action of all Atkin-Lehner involutions.
It can be computed by the following formula \cite[Section I.2, Proposition 1]{GKZ}: If $\Delta=\Delta_1\Delta_2$ is a factorization of $\Delta$ into discriminants and $N=N_1N_2$ is a factorization of $N$ into positive factors such that $(\Delta_1,N_1a)=(\Delta_2,N_2c)=1$, then
\begin{equation}\label{def:chi_Delta}
 \chi_{\Delta}(\left[a,b,Nc\right])=\left(\frac{\Delta_1}{N_1a}\right)\left(\frac{
\Delta_2}{N_2c}\right).
\end{equation}

If no such factorizations of $\Delta$ and $N$ exist, we have $\chi_{\Delta}(\left[a,b,Nc\right])=0$.

We note that $\chi_{\Delta}(\delta)$ depends only on $\delta \in L'$ modulo $\Delta L$. 
Therefore, we can view it as a function on the discriminant group $\dgdelta$.

Using the function $\chi_\Delta$ we obtain an intertwiner of the Weil representations
corresponding to $\dg$ and $\dgdelta$ as follows. 

\begin{definition}
 Denote by $\pi: \dgdelta \rightarrow \dg$ the canonical projection.
 For $h \in \dg$, we define
 \begin{equation}
  \psi_{\Delta,r}(\e_h) := \sum_{\substack{\delta \in \dgdelta \\ \pi(\delta)=rh \\ Q_\Delta(\delta) \equiv \sgn(\Delta) Q(h)\, (\Z)}} \chi_\Delta(\delta) \e_\delta.
 \end{equation}
\end{definition}

\begin{proposition} \label{prop:intertwiner}
 Let $\Delta \in \Z$ be a discriminant and $r \in \Z$ such that $\Delta \equiv r^2\ (4N)$. Then the map $\psi_{\Delta,r}: \dg \rightarrow \dgdelta$ defines an intertwining linear map between the representations $\widetilde{\rho}$ and $\rho_\Delta$, where $\widetilde{\rho} = \rho$ if $\Delta>0$ and $\widetilde{\rho}=\bar\rho$ if $\Delta<0$.
\end{proposition}

\begin{proof}
 For $T \in \Mp_2(\Z)$ this is trivial. For the generator $S$, this follows from Proposition 4.2 in \cite{BrOno}; namely, we have the following identity of exponential sums
\begin{equation}
 \sum_{\substack{\delta\in \dgdelta\\ \pi(\delta) = rh \\Q_\Delta(\delta)\equiv\sgn(\Delta) Q(h)\, (\Z)}} \hspace{-9mm}\chi_{\Delta}(\delta)\ e\left(\frac{(\delta,\delta')}{\abs{\Delta}}\right)
 = 
\epsilon \left| \Delta \right|^{3/2} \chi_{\Delta} (\delta')\ \hspace{-9mm}\sum\limits_{\substack{h'\in \dg\\ \pi(\delta')= rh' \\Q_\Delta(\delta')\equiv\sgn(\Delta) Q(h')\, (\Z)}}\hspace{-9mm} e\left(\sgn(\Delta)(h,h')\right),
\end{equation}
where $\epsilon = 1$ if $\Delta>0$ and $\epsilon=i$ if $\Delta<0$.
\end{proof}
This, together with the unitarity of the Weil representation, directly implies the following
\begin{corollary} \label{cor:twisting_mf}
 Let $f \in A_{k,\rho_\Delta}$. Then the function $g: \h \rightarrow \C[\dg]$, ${g=\sum_{h \in \dg} g_h \e_h}$
 with $  {g_h := \left\langle \psi_{\Delta,r}(\e_h), f \right\rangle}$, 
is contained in $A_{k,\widetilde{\rho}}$.
\end{corollary}


\section{The twisted Kudla-Millson theta function}\label{theta}
We let $\delta \in \dgdelta$ and define the 
same theta function $\Theta_{\delta}(\tau,z)$ for 
$\tau, z \in \h$ 
as Bruinier and Funke \cite{BrFu06}, here for the lattice $\Delta L$ with the quadratic form $Q_\Delta$.
It is constructed using the Schwartz function
\[
 \varphi^0_{\Delta}(\lambda,z)
 =\left(\frac{1}{\abs{\Delta}}(\lambda,\lambda(z))^2-\frac{1}{2\pi} \right) e^{-2\pi R(\lambda,z)/|\Delta|}\omega,
\]
where 
$R(\lambda,z):=\frac{1}{2}(\lambda,\lambda(z))^2-(\lambda,\lambda)$ 
and $\omega = \frac{i}{2}\frac{dz \wedge d\bar{z}}{y^2}$.

Then let $\varphi(\lambda,\tau,z)=e^{2\pi i Q_\Delta(\lambda)\tau} \varphi^0_{\Delta}(\sqrt{v}\lambda,z)$ and define
\begin{equation}\label{ThetaDeltaOp}
	\Theta_{\delta}(\tau,z,\varphi)=\sum\limits_{\lambda\in \Delta L+\delta} \varphi (\lambda,\tau, z).
\end{equation}

Since we are only considering the fixed Schwartz function $\varphi$ above, 
we will frequently drop the argument $\varphi$ and simply write $\Theta_{\delta}(\tau,z)$.

The Schwartz function $\varphi$ has been constructed by Kudla and Millson \cite{KM86}.
It has the crucial property that
for $Q(\lambda)>0$ it is a Poincar\'e dual form for the Heegner point $D_\lambda$,
while it is exact for $Q(\lambda)<0$.

The vector valued theta series
\[
\Theta_{\dgdelta}(\tau,z)=\sum\limits_{\delta\in \dgdelta} \Theta_{\delta}(\tau,z) \mathfrak{e}_{\delta}.
\]
is a $C^{\infty}$-automorphic form of weight $3/2$ which transforms with respect to the representation $\rho_\Delta$ \cite{BrFu06}.

Following Bruinier and Ono \cite{BrOno} we also define a twisted theta function.
For $h \in \dg$ the corresponding component is defined as
\begin{equation}\label{eq:twtheta}
 \Theta_{\Delta,r,h}(\tau,z) = \left\langle \psi_{\Delta,r}(\e_h), \overline{\Theta_{\dgdelta}(\tau,z)} \right\rangle =
  \sum\limits_{\substack{\delta\in \dgdelta\\ \pi(\delta) = rh \\Q_\Delta(\delta)\equiv\sgn(\Delta)Q(h)\, (\Z)}} \chi_{\Delta}(\delta)\Theta_{\delta}(\tau,z).
\end{equation}

Using this, we obtain a $\C[\dg]$-valued theta function by setting
\begin{equation}
	\Theta_{\Delta,r}(\tau,z) := \sum_{h \in \dg} \Theta_{\Delta,r,h}(\tau,z) \e_h.
\end{equation}
\begin{remark}
Note that Bruinier and Ono actually introduce their twisted Siegel theta function in a more direct way. 
However, our interpretation makes it possible to apply this ``method of twisting'' 
directly to other modular forms and theta kernels.

\end{remark}
By Proposition \ref{cor:twisting_mf} we obtain the following transformation formula for $\Theta_{\Delta,r}(\tau,z)$.
\begin{proposition} \label{prop:twthetatrans}
The theta function $\Theta_{\Delta,r}(\tau,z)$ is a non-holomorphic $\C[\dg]$-valued modular form of weight
$3/2$ for the representation $\widetilde{\rho}$. 
Furthermore, it is a non-holomorphic automorphic form of weight 0 for $\G_0(N)$ in the variable $z \in D$.
\end{proposition}
\begin{proof}
	In general $\Gamma_0(N)$ does not act trivially on $\dgdelta$. 
	However, the $\G_0(N)$ invariance of $\chi_\Delta$ implies
	\begin{equation*}
	     \Theta_{\Delta,r}(\tau,\gamma z)
		= \sum_{h \in \dg} \left\langle \sum_{\substack{\delta \in \dgdelta \\ \pi(\delta)=rh \\ Q_\Delta(\delta) \equiv \sgn(\Delta) Q(h)\, (\Z)}}\!\!\!\!\!\!\!
		\chi_\Delta(\gamma^{-1}\delta) \e_\delta,\ \overline{\Theta_{\dgdelta}(\tau,z)} \right\rangle
		= \Theta_{\Delta,r}(\tau, z).
	\end{equation*}
\end{proof}
%
%
\section{The twisted theta lift}\label{sec:mainresult}
\subsection{The twisted modular trace function}
Before we define the theta lift, we introduce a generalization 
of the twisted modular trace function given in the Introduction.
The twisted Heegner divisor on $M$ is defined by
\[
Z_{\Delta,r}(h,m)= \sum\limits_{\lambda \in \Gamma \backslash L_{rh,m\abs{\Delta}}}\frac{\chi_{\Delta}(\lambda)}{\left|\overline\G_{\lambda}\right|} Z(\lambda)\in \mathrm{Div}(M)_\Q.
\]
Note that for $\Delta=1$, we obtain the usual Heegner divisors \cite{BrFu06}.
Let $f$ be a harmonic weak Maass form of weight $0$ in $H^{+}_0(\G)$.

\begin{definition} 
If $m\in\Q_{>0}$ with $m \equiv \sgn(\Delta)Q(h)\ (\Z)$ and $h\in \dg$ we put
\begin{equation}\label{def:trace1}
\mt_{\Delta,r}(f;h,m) = \sum\limits_{z\in Z_{\Delta,r}(h,m)}f(z)=\sum\limits_{\lambda\in \G\setminus L_{rh,\abs{\Delta}m}} \frac{\chi_{\Delta}(\lambda)}{\left|\overline\G_{\lambda}\right|} f(D_{\lambda}).
\end{equation}
\end{definition}

\begin{definition}
If $m=0$ or $m\in\Q_{<0}$ is not of the form $\frac{-Nk^2}{\abs{\Delta}}$ with $k\in\Q_{>0}$ we let
\[
\mt_{\Delta,r}(f;h,m) = \begin{cases}
                               	-\frac{\delta_{h,0}}{2\pi}
                               	\int_{\G\backslash\h}^{\text{reg}}f(z) \frac{dxdy}{y^2}, &\text{if } \Delta = 1 \\
                               	0, &\text{if } \Delta \neq 1.
                               \end{cases}
\]

\noindent Here the integral has to be regularized \cite[$(4.6)$]{BrFu04}.

\noindent Now let $m = -Nk^2/\abs{\Delta}$ with $k\in\Q_{>0}$ and $\lambda\in L_{rh,m\abs{\Delta}}$. We have $Q(\lambda) = -Nk^2$, which implies that $\lambda^{\perp}$ is split over $\Q$ and $c(\lambda)$ is an infinite geodesic.
Choose an orientation of $V$ such that
\[
 \sigma_{\ell_{\lambda}}^{-1}\lambda= \begin{pmatrix} m & s \\  0 & -m \end{pmatrix}
\]
for some $s\in\Q$. Then $c_{\lambda}$ is explicitly given by
\[
 c_{\lambda}= \sigma_{\ell_{\lambda}} \left\{ z \in \h; \Re(z)=-s/2m\right\}.
\]
Define the real part of $c(\lambda)$ by $\Re(c(\lambda))=-s/2m$. For a cusp $\ell_{\lambda}$ let 
\begin{align*}
\langle f, c(\lambda)\rangle &= -\sum\limits_{n\in\Q_{<0}}a^{+}_{\ell_{\lambda}}(n)e^{2\pi i\Re(c(\lambda))n} -\sum\limits_{n\in\Q_{<0}} a^{+}_{\ell_{-\lambda}}(n)e^{2\pi i\Re(c(-\lambda))n}.
\end{align*}
Then we define
\begin{equation}\label{def:trace4}
\mt_{\Delta,r}(f;h,m)= \sum\limits_{\lambda\in \G\setminus L_{rh,\abs{\Delta}m}} \chi_{\Delta}(\lambda) \langle f,c(\lambda)\rangle.
\end{equation}
\end{definition}
In order to describe the coefficients of the lift that are not given in terms of traces we introduce the following definitions.
For $h \in \dg$, we let
\begin{equation}\label{def:deltah}
 \delta_\ell(h)=
 \begin{cases}
   1, & \text{if } \ell \cap (L+h) \neq \emptyset,
   \\
   0, & \text{otherwise.}
 \end{cases}
\end{equation}
If $\delta_\ell(h)=1$, there is an $h_\ell$ such that $\ell \cap (L+h)=\Z \lambda_\ell + h_\ell$. Now let $s \in \Q$ such that $h_\ell = s \lambda_\ell$. Write $s=\frac{p}{q}$ with $(p,q)=1$ and define $d(\ell,h):=q$, which depends only on $\ell$ and $h$. Moreover, we define $h'_\ell=\frac{1}{d(\ell,h)}\lambda_\ell$ which is well defined as an element of $\dg$.

\begin{definition} \label{def:dzero}
	Let $h \in \dg$ and $\ell \in \Iso(V)$. Then we let
	\begin{equation}
		\dlconst_{\Delta,r}(\ell,h) := \begin{cases}
						\delta_\ell(h), &\text{if } \Delta=1,\\
		                           	\chi_\Delta((rh)'_\ell), &\text{if }
								    \Delta \neq 1,\ \delta_\ell(rh)=1 \text{ and } \Delta \mid d(\ell,rh),\\
		                           	0, &\text{otherwise.}
		                           \end{cases}
	\end{equation}
\end{definition}

If $\Delta \mid d(\ell,rh)$, then $\chi_\Delta((rh)'_\ell)$ is well defined because $(rh)'_\ell \in L'$. The fact that $d(\ell,0)=1$ implies that for $\Delta\neq 1$ and $rh=0$ we have  $\dlconst_{\Delta,r}(\ell,0)=0$.

\begin{proposition} \label{prop:dzero}
    Assume that $\Delta \neq 1$.
    If $\dlconst_{\Delta,r}(\ell,h) \neq 0$, then for every prime $p \mid \Delta$, we have that $p^2 \mid N$.
    In particular, for square-free $N$ and $\Delta \neq 1$, we always have $\dlconst_{\Delta,r}(\ell,h) = 0$.
\end{proposition}
\begin{proof}
	We can assume that $rh \neq 0$ because otherwise $\dlconst_{\Delta,r}(\ell,h) = 0$. Since $\lambda_\ell$ is primitive it is of the form
	$\lambda_\ell = \left(\begin{smallmatrix}
	         	b & -\frac{\Delta a}{N} \\ \Delta c & -b
	         \end{smallmatrix}\right)$,
	with $a,b,c \in \Z$, $b\neq 0$, $(a,b,c)=1$ and $(\Delta,b)=1$. 
	The facts that $\frac{1}{d(\ell,rh)}\lambda_\ell \in L'$ and $\Delta \mid d(\ell,rh)$ imply $\Delta \mid 2N$. Since $\lambda_\ell$ is isotropic, we have $\frac{\Delta^2 ac}{N} = b^2 \in \Z$. Now suppose that there exists a prime $p$ such that $p \mid \Delta$ but $p^2 \nmid N$. Then $\frac{\Delta^2 ac}{Np} \in \Z$ and thus $p \mid b$ which is a contradiction.
\end{proof}

\subsection{The theta integral}
Now we consider the integral
\[
 I_{\Delta,r}(\tau,f)= \int_M f(z) \Theta_{\Delta,r}(\tau,z) = \sum_{h\in\dg}\left(\int_M f(z) \Theta_{\Delta,r,h}(\tau,z)\right)\mathfrak{e}_h,
\]
where $\Delta$ and $r$ are chosen as before.
For the individual components, we write
\[
  I_{\Delta,r,h}(\tau,f)=\int_M f(z) \Theta_{\Delta,r,h}(\tau,z).
\]
This is a twisted version of the theta lift considered by Bruinier and Funke \cite{BrFu06}, 
which we obtain as a special case when $\Delta=1$.
Note that due to the rapid decay of the Kudla-Millson kernel the integral $I_{\Delta, r}(\tau,f)$ 
converges \cite[Proposition 4.1]{BrFu06}. It defines a harmonic weak Maass form of weight $3/2$ transforming with the representation $\widetilde{\rho}$.

\begin{theorem}\label{thm:main}
Let $f\in H^{+}_{0}(\G)$ with Fourier expansion as in \eqref{eq:fourierhmf0}.
Assume that $f$ has vanishing constant term at every cusp of $\G$. Then the Fourier expansion of $I_{\Delta,r, h}(\tau,f)$ is given by
\begin{equation}\label{thm:felift}
  I_{\Delta,r,h}(\tau,f)
   = \sum_{\substack{m\in\Q_{>0}\\ m\equiv \sgn(\Delta) Q(h)\, (\Z)}}\!\!\!\!\!\! \mt_{\Delta,r}(f;h,m)q^m
   \ + \!\!\!\!\!\!\!\! 
   \sum\limits_{\substack{m\in\Q_{>0} \\ -N \abs{\Delta} m^2 \equiv\sgn(\Delta) Q(h)\, (\Z)}}
   \!\!\!\!\!\!\!\!\!\!\!\!\!\! 
	 \mt_{\Delta,r}(f;h,-N\abs{\Delta}m^2) q^{-N \abs{\Delta} m^2}.
\end{equation}
If the constant coefficients of $f$ at the cusps do not vanish, the following terms occur in addition:
\begin{gather*}
    \frac{\sqrt{\abs{\Delta}}}{2\pi\sqrt{Nv}}
    \sum_{\ell \in \Gamma\backslash \Iso(V) }
    \hspace{-2mm} \dlconst_{\Delta,r}(\ell,h)\,\epsilon_\ell\, a^+_{\ell}(0)\\
    + \sqrt{\abs{\Delta}} \sum_{m>0}\sum_{\substack{\lambda \in \Gamma \backslash L_{rh,-Nm^2} \\ Q_\Delta(\lambda)\equiv\sgn(\Delta)Q(h)\, (\Z)}} \chi_\Delta(\lambda) \frac{a^{+}_{\ell_{\lambda}}(0)+a^{+}_{\ell_{-\lambda}}(0)}{8\pi\sqrt{vN}m} \beta\left(\frac{4\pi vNm^2}{|\Delta|}\right)q^{-Nm^2/|\Delta|},
\end{gather*}
where $\beta(s)=\int_1^{\infty}t^{-3/2}e^{-st}dt$.

In particular, $I_{\Delta,r,h}(\tau,f)$ is contained in $H_{3/2,\widetilde{\rho}}$.
\end{theorem}

\begin{remark}
 If $N$ is square-free, Proposition \ref{prop:dzero} implies that $I_{\Delta,r, h}(\tau,f)\in H^{+}_{3/2,\widetilde{\rho}}$.
\end{remark}

Since the traces of negative index essentially depend on the principal part of $f$, it is not hard to show the following
\begin{corollary}
 Assume that all constant coefficients of $f\in H^{+}_0(\G)$ vanish; then
\[
 I_{\Delta, r}(\tau,f)\in M^{\text{!}}_{3/2,\widetilde{\rho}}.
\]
\end{corollary}

\begin{proof}[Proof of Theorem \ref{thm:main}]
The theorem can be proven in two different ways. 
It is possible to give a proof by explicitly calculating
the contributions of the lattice elements of positive, 
negative, and vanishing norm, similarly to Bruinier and Funke \cite{BrFu06}.
However, a substantially shorter proof is obtained by rewriting 
the twisted theta integral as a linear combination of untwisted ones. 
Throughout the proof, we may assume that $\Delta \neq 1$ 
because for $\Delta = 1$ the statement is covered by Bruinier and Funke \cite{BrFu06}.

Replacing the theta function $\Theta_{\Delta,r}(\tau,z)$ by the expression in \eqref{eq:twtheta}, we can write
\[
 I_{\Delta,r}(\tau,f) = \int_{\mathcal{F}} f(z) \Theta_{\Delta,r}(\tau,z) 
 = \sum_{h \in \dg} \left\langle \psi_{\Delta,r}(\e_h), \int_{\mathcal{F}} f(z) \overline{\Theta_{\dgdelta}(\tau,z)} \right\rangle \e_h,
\]
where $\mathcal{F}$ denotes a fundamental domain for the action of $\G$ on $\h$.

In general $\G$ does not act trivially on $\dgdelta$. But $\Theta_{\dgdelta}(\tau,z)$ is always invariant under the discriminant kernel
$\G_\Delta = \left\lbrace \gamma \in \G;\ \gamma \delta = \delta \text{ for all } \delta \in \dgdelta \right\rbrace \subset \G$.
Since $f(z)\Theta_{\Delta,r}(\tau,z)$ is $\G$-invariant by Proposition \ref{prop:twthetatrans}, we obtain by a standard argument
\[
	I_{\Delta,r}(\tau,f) = \frac{1}{[\G:\G_\Delta]}\ \sum_{h \in \dg} \left\langle \psi_{\Delta,r}(\e_h), \int_{\G_\Delta\backslash\h} f(z) \overline{\Theta_{\dgdelta}(\tau,z)} \right\rangle.
\]
Now we are able to apply the result of Bruinier and Funke \cite[Theorem 4.5]{BrFu06} to the integral above. 
For $m \in \Q$ we obtain that the $(h,m)$-th Fourier coefficient of the holomorphic part 
of $I_{\Delta,r, h}(\tau,f)$ is given by $1/[\G:\G_\Delta]$ times
\begin{equation}\label{eq:tracelincomb}
	\left\langle \psi_{\Delta,r}(\e_h), \sum_{\delta \in \dgdelta} \overline{\mt(f;\delta,m)} \e_\delta\right\rangle 
	= \sum_{\substack{\delta \in \dgdelta \\ \pi(\delta)=rh \\ Q_\Delta(\delta) \equiv \sgn(\Delta) Q(h)\, (\Z)}} \chi_\Delta(\delta)\ \mt(f;\delta,m).
\end{equation}
Here the traces are taken with respect to $\G_\Delta$ and the discriminant group $\dgdelta$.
Note that
\begin{equation*}
	\mt(f;\delta,m) =
	\begin{cases}
	\sum\limits_{\G_\Delta \backslash (\Delta L)_{\delta,m}} \frac{1}{\abs{\overline{\G}_{\Delta,\lambda}}} f(D_\lambda), & \text{if } m>0,
	\\
	\sum\limits_{\G_\Delta \backslash (\Delta L)_{\delta,m}} \langle f,c(\lambda)\rangle, & \text{if } m<0,
	\end{cases}
\end{equation*}
where $(\Delta L)_{\delta,m}=\left\lbrace \lambda\ \in\ \Delta L + \delta;\ Q_\Delta(\lambda) = m \right\rbrace$.
For $m=0$, $\mt(f;0,0)$ is defined as a regularized integral and we have $\mt(f;\delta,m)=0$ for $\delta \neq 0$. 
Hence for $m = 0$ the left hand side in \eqref{eq:tracelincomb} is equal to
$\chi_\Delta(0)\ \mt(f;0,0)$. Since $\chi_\Delta(0)=0$ for $\Delta \neq 1$, this quantity vanishes in our case.

If $m \equiv \sgn(\Delta)Q(h) \pmod{\Z}$ the right hand side in \eqref{eq:tracelincomb} is equal to
$\mt_{\Delta,r}(f;h,m)$ as in \eqref{def:trace1} and \eqref{def:trace4}.  Otherwise it vanishes.

Next, we consider the non-holomorphic part. 
It is again a straightforward calculation to obtain our result for the coefficients of negative index.
It remains to evaluate
\begin{equation}\label{eq:proofmain}
\frac{1}{2\pi \sqrt{N\abs{\Delta} v}}
 \left\langle \psi_{\Delta,r}(\e_h), 
	\sum_{\delta \in \dgdelta}
	\sum\limits_{\substack{\ell\in \G\setminus\Iso(V)\\ \ell\cap \Delta L + \delta \neq \emptyset}}
	\chi_\Delta(\delta) a_\ell^+(0) \varepsilon_\ell \right\rangle.
\end{equation}
Recall the definition of $\delta_{\ell}(rh)$ from \eqref{def:deltah}. For $\delta_{\ell}(rh)=1$
there is an element $(rh)_\ell \in L + rh$ such that $\ell \cap (L + rh) = \Z \lambda_\ell + (rh)_\ell$, where
$\lambda_\ell$ is a primitive element of $\ell \cap L$.
Consequently, a system of representatives for all $\delta \in \dgdelta$ 
with $\pi(\delta)=rh$ such that $\ell \cap (\Delta L + \delta) \neq \emptyset$
is given by the set $\{ m\lambda_\ell + rh;\ m \mod \Delta \}$.
So it boils down to computing
\[
 \delta_{\ell}(rh)\sum\limits_{m \bmod \Delta} \chi_{\Delta}(m\lambda_\ell+(rh)_{\ell}).
\]
To do this, we write $(rh)_\ell=\frac{n}{d(\ell,rh)}\lambda_\ell$ for some integer $n$ and $(rh)'_\ell=\frac{1}{d(\ell,rh)}\lambda_\ell$. So the inner product in \eqref{eq:proofmain} equals
\begin{align*}	 
	 \sum_{m \bmod \Delta} \chi_\Delta\left(\frac{d(\ell,rh)m+n}{d(\ell,rh)}\lambda_\ell\right)
	= \chi_\Delta((rh)'_\ell) \sum_{m \bmod \Delta} \Deltaover{d(\ell,rh)m+n}.
\end{align*}
The latter sum vanishes unless $\Delta$ divides $d(\ell,rh)$, in which case it equals $\abs{\Delta}$.
\end{proof}

Similarly to Bruinier and Funke \cite{BrFu06}, we can give a more explicit description of the traces
of negative index $\mt_{\Delta,r}(f;h,-Nk^2/\abs{\Delta})$.
For this, sort the infinite geodesics according to the cusps from which they originate:
For $k \in \Q_{>0}$, define $L_{h,-Nk^2,\ell}=\{ X \in L_{h,-Nk^2};\, X \sim \ell \}$ and let
\begin{equation}
	\nu_\ell(h,-Nk^2) := \#\Gamma_\ell\backslash L_{h,-Nk^2,\ell}.
\end{equation}
We have that $\nu_\ell(h,-Nk^2) = 2k\epsilon_\ell$ if $L_{h,-Nk^2,\ell}$ is non-empty.

Let $h \in \dg$, $n \in \Z$ and $m=-Nk^2/\abs{\Delta}$ with $k \in \Q_{>0}$, such that $\Delta \mid 2k\varepsilon_\ell$ and $\frac{2 k \varepsilon_\ell}{\abs{\Delta}} \mid n$. We define
\begin{equation}\label{def:muell}
	\mu_\ell(h,m,n) =
	\frac{\nu_\ell(h,-Nk^2)}{\sqrt{\abs{\Delta}}}
	\overline{\epsilon}
	\sum_{\substack{j \bmod \Delta \\ N\beta_\ell j - n'\, \equiv\, 0 \bmod \Delta}}
	\Deltaover{j}\, \exp\left({\frac{4\pi i Nkr_\ell j}{\abs{\Delta}}}\right).
\end{equation}
Here, we let $r_\ell = \Re(c(\lambda))$ for any $\lambda \in L_{h,-Nk^2,\ell}$ and $n' = \frac{\abs{\Delta}}{2k\varepsilon_\ell}n$.
Moreover, $\epsilon=1$ for $\Delta > 0$ and $\epsilon=i$ for $\Delta < 0$.

\begin{remark}
        \label{rem:explform}
	The finite sum in \eqref{def:muell} can often be explicitly evaluated. For instance, 
	if $N\beta_\ell$ is coprime to $\Delta$, it is equal to
	\begin{equation*}
		\Deltaover{N\beta_\ell n'}\, \exp\left({\frac{4 \pi i N k r_\ell n' s}{\abs{\Delta}}}\right),
	\end{equation*}
	where $s$ denotes an integer such that $(N\beta_\ell) s \equiv 1 \bmod \Delta$.
\end{remark}

\begin{proposition}\label{prop:tf-}
	    Let $f \in H^{+}_{0}(\G)$ with Fourier expansion as in \eqref{eq:fourierhmf0} and let $m=\frac{-N {k}^2}{\abs{\Delta}}$ for some $k \in \Q_{>0}$. Then $\mt_{\Delta,r}(f;h,m)=0$ unless $k=\frac{\abs{\Delta} k'}{2N}$ for some $k' \in \Z_{>0}$. In the latter case, we have
	    \begin{align*}
		    \mt_{\Delta,r}(f;h,m) &= - \sum_{\ell\in\G\backslash\Iso(V)}
			   \sum_{n \in \frac{2k}{\abs{\Delta}\beta_\ell}\Z_{<0}}
			  a^+_\ell(n)\, \mu_\ell(rh,m,n\, \alpha_\ell)\, e^{-2\pi i r_\ell n}\\
			  &\quad - \sgn(\Delta)\sum_{\ell\in\G\backslash\Iso(V)}
			  \sum_{n \in \frac{2k}{\abs{\Delta}\beta_\ell}\Z_{<0}}
			  a^+_\ell(n)\, \mu_\ell(-rh,m,n\, \alpha_\ell)\, e^{-2\pi i r'_\ell n},
	    \end{align*}
with $r_\ell = \Re(c(\lambda))$ for any $\lambda \in L_{rh,\abs{\Delta}m,\ell}$ and $r'_\ell = \Re(c(\lambda))$ for any $\lambda \in L_{-rh,\abs{\Delta}m,\ell}$.
 
\noindent Moreover, we have $\mt_{\Delta,r}(f;h,-Nn^2/\abs{\Delta}) = 0$ for $n \gg 0$.
\end{proposition}
For the proof of the proposition we need the following
\begin{lemma}\label{lem:gmdelta}
	For a fundamental discriminant $\Delta$, integers $a,b,n \in \Z$ and $M \in \Z_{>0}$, 
	such that $\Delta \mid M$, we define the Gauss type sum
	\begin{equation*}
		g_M^\Delta(a,b;n) = \sum_{j \bmod M} \Deltaover{aj+b}\, e^{\frac{2 \pi i j n}{M}}.
	\end{equation*}
	The sum $g_M^\Delta(a,b;n)$ vanishes unless $\frac{M}{\Delta} \mid n$. Then we obtain
     \begin{equation*}
     	g_M^\Delta(a,b;n) = \epsilon^{-1} \frac{M}{\sqrt{\abs{\Delta}}}
					\sum_{\substack{l \bmod \Delta \\ al+n' \equiv 0 \bmod \Delta}}
					\Deltaover{l}\, e^{\frac{2\pi i b l}{\abs{\Delta}}},
     \end{equation*}
     where $n':=\frac{\abs{\Delta}}{M}n$.
\end{lemma}
\begin{remark}
	In many cases the sum above can also be evaluated more explicitly in a straightforward but tedious way.
\end{remark}
\begin{proof}
    Replacing the Kronecker symbol by the finite exponential sum
	\begin{equation*}
		\Deltaover{aj+b} = \epsilon^{-1} \frac{1}{\sqrt{\abs{\Delta}}}
		\sum_{l \bmod \Delta} \Deltaover{l}\, e^{\frac{2\pi i (aj+b) l}{\abs{\Delta}}}
	\end{equation*}
	yields
	 \begin{align*}
     	g_M^\Delta(a,b;n)
				&= \epsilon^{-1} \frac{1}{\sqrt{\abs{\Delta}}}
				    \sum_{l \bmod \Delta} \Deltaover{l}\, e^{\frac{2\pi i b l}{\abs{\Delta}}}
				    \sum_{j \bmod M} \exp\left({\frac{2 \pi i j \left(a\frac{M}{\abs{\Delta}}l +n \right)}{M}}\right)\\
				&=	\epsilon^{-1} \frac{M}{\sqrt{\abs{\Delta}}}
					\sum_{\substack{l \bmod \Delta \\ a\frac{M}{\abs\Delta}l+n\, \equiv\, 0 \bmod M}}
					\Deltaover{l}\, e^{\frac{2\pi i b l}{\abs{\Delta}}}.
     \end{align*}
     The congruence condition above can only be satisfied if $\frac{M}{\Delta} \mid n$, which proves the statement of the lemma. \qedhere 
\end{proof}

\begin{proof}[Proof of Proposition \ref{prop:tf-}]
    Following Bruinier and Funke \cite[Proposition 4.7]{BrFu06}, we write
    \begin{equation*}
    	\mt_{\Delta,r}(f;h,m) = \sum_{\ell\in\G\backslash\Iso(V)} G_\Delta(rh,-Nk^2,\ell)
    	+ \sgn(\Delta)\sum_{\ell\in\G\backslash\Iso(V)} G_\Delta(-rh,-Nk^2,\ell),
    \end{equation*}
    where
    \begin{equation*}
    	 G_\Delta(h,-Nk^2,\ell) = - \sum_{\lambda\in \Gamma_\ell \backslash L_{h,-Nk^2,\ell}}
		\chi_\Delta(\lambda) \sum_{n \in \Z_{<0}} a^+_\ell(n/\alpha_\ell)\, e^{\frac{2\pi i \Re(c(\lambda))n}{\alpha_\ell}}.
    \end{equation*}
     
    A set of representatives for $\Gamma_\ell \backslash L_{h,-Nk^2,\ell}$ is given by
    \[
     \left\{Y_j=\sigma_\ell\begin{pmatrix}k&-2kr_\ell-j\beta_\ell\\0&-k\end{pmatrix};j=0,\ldots,2k\epsilon_\ell-1 \right\}
    \]
    for some $r\ell\in\Q$.
    We have $\Re(c(Y_j))=r_\ell+j\frac{\beta_\ell}{2k}$.
    
    For $\lambda \in L_{h,-Nk^2,\ell}$ and $k$ not of the form $\frac{\abs{\Delta} k'}{2N}$ for some $k' \in \Z_{>0}$
    we have $\chi_\Delta(\lambda)=0$. So we may assume that $k=\frac{\abs{\Delta} k'}{2N}$.
    
    By reordering the summation and using the $\SO^+(L)$-invariance of $\chi_\Delta$ 
    we see that it remains to evaluate
    \begin{equation*}
    	\sum_{j=0}^{2k\varepsilon_\ell-1}\Deltaover{N\beta_\ell j + 2 N k r_\ell}\, e^{\frac{-2\pi i nj}{2k\varepsilon_\ell}}
    \end{equation*}
    for $n \in \Z_{<0}$.
    Using Lemma \ref{lem:gmdelta}, we obtain the statement of the proposition.
\end{proof}

\newpage
\section{Applications and Examples}
\subsection{The twisted lift of the weight zero
            Eisenstein series and ${\log\lVert\Delta(z)\rVert}$}\label{sec:Eisenlift}
For $z\in\h$ and $s\in\C$, we let
\[
 \mathcal{E}_0(z,s)=\frac12\zeta^{*}(2s+1)
 \sum\limits_{\gamma\in\G_{\infty}\setminus \SL_2(\Z)}
 \left(\Im(\gamma z)\right)^{s+\frac12}.
\]
Here $\zeta^*(s) = \pi^{-s/2}\Gamma(s/2)\zeta(s)$ is the completed Riemann Zeta function
and $\G_\infty$ is the stabilizer of the cusp $\infty$ in $\G$.
We now consider the case $N=1$, so we have the quadratic form $Q(\lambda)=\mathrm{det}(\lambda)$ and the lattice
\[
 L=\left\{\begin{pmatrix} b&a\\c&-b\end{pmatrix}; a,b,c\in\Z\right\}.
\]
Furthermore, we let $K$ be the one-dimensional lattice $\Z$ together with the negative definite bilinear form $(b,b')=-2bb'$.
Then we have $L'/L \simeq K'/K$. 
We define a vector valued Eisenstein series $\mathcal{E}_{3/2,K}(\tau,s)$ of weight $3/2$ for the representation $\rho_K$ by
\[
 \mathcal{E}_{3/2,K}(\tau,s)=-\frac{1}{4\pi}\left(s+\frac12\right)\zeta^{*}(2s+1)\sum\limits_{\gamma'\in\G'_{\infty}\setminus \G'}\left.(v^{\frac12\left(s-\frac12\right)}\mathfrak{e}_0)\right|_{3/2,K}\gamma',
\]
where $\G'_{\infty}$ and $\G'$ are the inverse images of $\G_{\infty}$
and $\G$ in $\Mp_2(\Z)$.
Note that
\[
 \mathcal{F}(\tau,s):=\left(\mathcal{E}_{3/2,K}(4\tau,s)\right)_0+\left(\mathcal{E}_{3/2,K}(4\tau,s)\right)_1
\]
evaluated at $s=\frac12$ is equal to Zagier's Eisenstein series as in \cite{HiZa} and \cite{Zagierclass}.

Similarly to \cite{BrFu06}, Section 7.1, one can show the following theorem.
\begin{theorem}
 Assume that $\Delta > 0$. We have
\[
 I_\Delta(\tau,\mathcal{E}_0(z,s))=\Lambda\left(\varepsilon_\Delta, s+\frac12\right)\mathcal{E}_{3/2,K}(\tau,s).
\]
Here $\varepsilon_\Delta(n)=\left(\frac{\Delta}{n}\right)$ and $\Lambda\left(\varepsilon_\Delta, s+\frac12\right)$ denotes the completed Dirichlet $L$-series associated with $\varepsilon_\Delta$. Note that for $\Delta<0$ the lift $I_\Delta$ vanishes.

Using this, we obtain
\[
 I_\Delta(\tau,1)= \begin{cases}
                   	0, &\text{if } \Delta \neq 1,\\
                   	2\mathcal{E}_{3/2,K}\left(\tau,\frac12\right), &\text{if } \Delta = 1.
                   \end{cases}
\]
\end{theorem}

By $\Delta(z)=q\prod_{n=1}^\infty(1-q^n)^{24}$ we denote the Delta function. Following Bruinier and Funke we normalize the Petersson metric of $\Delta(z)$ such that
\[
 \lVert \Delta(z)\rVert =e^{-6C} \abs{\Delta(z)(4\pi y)^6},
\]
where $C=\frac12(\gamma+\log 4\pi)$. 

\begin{theorem}
	We have
	\begin{equation*}
	    -\frac{1}{12} I_\Delta(\tau,\log \lVert\Delta(z)\rVert) =
		\begin{cases}
			\Lambda(\varepsilon_\Delta,1)\mathcal{E}_{3/2,K}\left(\tau,\frac12\right) &\text{if } \Delta > 1,\\
			\mathcal{E}'_{3/2,K}\left(\tau,\frac12\right) &\text{if } \Delta=1.\\
		\end{cases}
	\end{equation*}
\end{theorem}

In terms of arithmetic geometry we obtain the following interpretation of this result (for notation and background information, we refer to the survey article by Yang \cite{Yang}). By $\mathcal{M}$ we denote the Deligne-Rapoport compactification of the moduli stack over $\Z$ of elliptic curves. Kudla, Rapoport and Yang \cite{KRY, Yang} construct cycles $\hat{\mathcal{Z}}(D,v)$ in the extended arithmetic Chow group of $\mathcal{M}$ for $D\in\Z$ and $v>0$. 
Then the Gillet-Soul\'e intersection pairing of $\hat{\mathcal{Z}}(D,v)$ with the normalized metrized Hodge bundle $\hat{\omega}$ on $\mathcal{M}$ is given in terms of the derivative of Zagier's Eisenstein series \cite{Yang, BrFu06}.

Similarly, one can define twisted cycles $\hat{\mathcal{Z}}_\Delta(D,v)$.
In contrast to the untwisted case, for $\Delta > 1$ the intersection pairing is given in terms of the value of the Eisenstein series at $s = 1/2$ (also note that the degree of the twisted divisor is $0$). We have
\[
 \sum\limits_{D \in \Z} \langle\hat{\mathcal{Z}}_\Delta(D,v), \hat{\omega} \rangle q^D = \log(u_\Delta)\; h(\Delta)\; \mathcal{F}\left(\tau,\frac12\right),
\]
where $u_\Delta$ denotes the fundamental unit and $h(\Delta)$ the class number of the real quadratic number field of discriminant $\Delta$.

\subsection{The twisted lift of Maass cusp forms}\label{sec:liftMC}
As indicated in the remark in Section \ref{theta} the construction of the twisted theta function directly yields a twisted version of other theta lifts. As an example we briefly consider the lift of Maass cusp forms analogously to Bruinier and Funke \cite[Section 7.2]{BrFu06}. It was first considered by Maass (\cite{Maass}, \cite{Duke}, \cite{KS}). The underlying theta kernel is now given by the Siegel theta series for $\Delta L$ with the  negative quadratic form $-Q_\Delta$; namely,
\[
 \Theta_{\delta}(\tau,z,\varphi_{2,1})=\sum\limits_{\lambda\in \delta+\Delta L}\varphi_{2,1}(\lambda,\tau,z),
\]
where $\varphi_{2,1}(\lambda,\tau,z)=v^{1/2}e^{\frac{\pi i}{\abs{\Delta}}(-u(\lambda,\lambda)+iv(\lambda,\lambda)_z) }$.

Following the construction in Section \ref{theta} we obtain a twisted lift $I^M_\Delta(\tau,f)$. Using that
 \[
  \xi_{1/2}\varphi_{2,1}(\lambda,\tau,z)\omega=-\pi\varphi(\lambda,\tau,z),
 \]
a straightforward calculation yields
\begin{theorem}[Twisted version of Theorem 7.7 in \cite{BrFu06}]
 We have 
 \[
  \xi_{1/2}I^M_{\Delta}(\tau,f)=-\pi I_{\Delta}(\tau,f)
 \]
and for an eigenfunction $f$ of $\Omega$ with eigenvalue $\alpha$
\[
 \xi_{3/2}I_{\Delta}(\tau,f)=-\frac{\alpha}{4\pi} I^M_{\Delta}(\tau,f).
\]
\end{theorem}

\subsection{The example $\G_0(p)$}
	We begin by explaining how to obtain the example in the Introduction. 
	Let $N=p$ be a prime and $\Delta > 1$ a fundamental discriminant with $(\Delta,2p)=1$
	such that there exists an $r \in \Z$ with $\Delta \equiv r^2 \bmod 4p$.
	
	Let $f \in M_0^!(\Gamma_0(p))$ be invariant under the Fricke involution with Fourier expansion $f(\tau)=\sum_n a(n) q^n$.
	
	The group $\G_0(p)$ only has the two cusps $\infty, 0$ which are represented by the isotropic lines
        $\ell_\infty=\text{span}\left(\begin{smallmatrix} 0& 1\\0&0\end{smallmatrix}\right)$ and
        $\ell_0=\text{span}\left(\begin{smallmatrix} 0& 0\\-1&0\end{smallmatrix}\right)$.
        We then obtain $\alpha_{\ell_\infty}=1, \beta_{\ell_\infty}=1/p, \epsilon_{\ell_\infty}=p$ and $\alpha_{\ell_0}=p, \beta_{\ell_0}=1, \epsilon_{\ell_0}=p$.
	The Fricke involution interchanges the two cusps.

     The space $M^!_{3/2,\rho}$ is isomorphic to $M^{!,+}_{3/2}(\Gamma_0(4p))$,
     the subspace of $M^{!}_{3/2}(\Gamma_0(4p))$ containing only forms whose $n$-th Fourier coefficient is zero unless $n$ is a square modulo $4p$.     
     The isomorphism takes $\sum_{h \in \dg} f_h \e_h$ to $\sum_{h \in \dg} f_h$ \cite[Theorem 5.6]{EZ}.
     The assumption ${(\Delta,2p)=1}$ guarantees that we can choose the parameter $r \in \Z$
     as a unit in $\Z/4p\Z$. Thus the sum $\sum_{h \in \dg} I_{\Delta,r,h}$ does not depend on $r$.

	As described in Section \ref{sec:twisting}, we can identify lattice elements with integral binary quadratic forms.
	The action of $\G_0(p)$ on both spaces is compatible.
	This way we obtain the interpretation of the coefficients of positive index of the holomorphic part.
	However, notice that we have to consider positive and negative definite quadratic forms. For positive $\Delta$
	we have $\chi_\Delta(-Q)=\chi_\Delta(Q)$, which for $m=d/4p>0$ yields
	\[
		\sum_{h \in \dg} \sum\limits_{\lambda\in \G\setminus L_{rh,\abs{\Delta}m}} 
		\frac{\chi_{\Delta}(\lambda)}{\left|\overline\G_{\lambda}\right|} f(D_{\lambda})
		= 2 \sum\limits_{Q\in\G\backslash\mathcal{Q}_{-d\abs{\Delta},N}}
		\frac{\chi_{\Delta}(Q)}{\abs{\overline\G_Q}}f(\alpha_Q).
	\]
	We use the explicit formula in Proposition \ref{prop:tf-} to determine the coefficients of negative index.
      	For every $k \in \Z_{>0}$ with $k \equiv h \bmod 2p$, we have that
        $\left(\begin{smallmatrix} -k/2p& 0\\0&k/2p\end{smallmatrix}\right) \in L_{-h,-k^2/4p,\ell_\infty}$
        and $\left(\begin{smallmatrix} k/2p& 0\\0&-k/2p\end{smallmatrix}\right) \in L_{h,-k^2/4p,\ell_0}$. 
	And if $h \neq 0$ we have $L_{h,-k^2/4p,\ell_\infty} = \emptyset$ and $L_{-h,-k^2/4p,\ell_0} = \emptyset$.
        In particular, this implies that $r_\ell=r'_{\ell}=0$ in Proposition \ref{prop:tf-} and we get
        \begin{equation*}
          \sum_{h \in \dg} \mt_{\Delta,r}(f;h,m) = - 2 \sum_{\ell\in\G\backslash\Iso(V)}
			   \sum_{n \in \Z_{<0}}
			  a^+_\ell\left(\frac{k}{\abs{\Delta}p\, \beta_\ell} n\right)
                          \sum_{h \in \dg} \mu_\ell\left(rh,m,\frac{k}{\abs{\Delta}}\, n\right).
        \end{equation*}
        Moreover, Proposition \ref{prop:tf-} implies that $k=\abs{\Delta}k'$ for some $k' \in \Z_{>0}$.
        By our considerations above, for every given $k'$ and $\ell$ there is exactly one $h$
        such that $\mu_{\ell}(rh,-\abs{\Delta}{k'}^{2}/4p, k' n) \neq 0$.
        Using the explicit formula given in Remark \ref{rem:explform}, we obtain
	\[
	   \sum_{h \in \dg} \mu_\ell\left(rh,-\frac{\abs{\Delta}{k'}^{2}}{4p}, k' n\right) = \sqrt{\abs{\Delta}} k' \Deltaover{n}.
	\]
	Since $f$ is invariant under the Fricke involution, this yields the formula for the principal part in the Introduction.
	
	Note that for $\Delta \neq 1$, we do not obtain a non-holomorphic part in this case, which follows from the formula given in Theorem \ref{thm:main}
        and the fact that $$\sum_{\lambda \in \Gamma_0(p)\backslash L_{h,m}} \chi_\Delta(\lambda)=0$$ for all $h \in \dg$ and all $m \in \Q$.
	
	The computation above is also valid for $p=1$, except that we only have to consider the cusp at $\infty$ and
	we do not have to assume that $(\Delta,2)$=1.

\subsection{Computations}
  We consider the genus 1 modular curve $X_0(11)$ and the weakly holomorphic modular function $f(z)=J(11z)$. 
  Let $\Delta=5$ and $r=7$.
  
  Consider the case $d=8$. The class number of $\Q(\sqrt{-40})$ is 2 and a set of representatives for $\G_0(11)\backslash L_{2/22, 40/44}$ is given by the integral binary quadratic forms $[1001, 200, 10]$, $[-1001, 200, -10]$, $[407, 90, 5]$, $[-407, 90, -5]$. For $\G_0(11)\backslash L_{20/22, 40/44}$ a set of representatives is given by the negatives of the above forms and for all $h \neq \pm 2/22$ the set $L_{h, 40/44}$ is empty.
  
  Using the Galois-theoretic interpretation of the twisted traces, we can calculate them ``by hand'': the CM point given by the form $[407, 90, 5]$ is $z_0=\frac{-90+\sqrt{-40}}{814}$ and the CM value $f(z_0) \approx 20641.38121$ is an algebraic integer of degree 2. It is in fact a root of the polynomial $x^2 - 425691312x + 8786430582336$ and is contained in the Hilbert class field $H_{-40}$ of $\Q(\sqrt{-40})$. So $f(z_0) = \frac{1}{2}\left(425691312 - 190356480\sqrt{5}\right)$.
 	Therefore, we have for the twisted trace (for positive definite forms with $b \equiv 2 \bmod 11$)
 	\[
 		\frac{1}{\sqrt{5}} \sum_{Q \in \Gamma_0(11) \backslash \mathcal{Q}_{-40,11,2}} \chi_{5}(Q)f(\alpha_Q)= 190356480.
 	\]
 	Since our definition of the trace includes both, positive and negative definite quadratic forms and $\chi_{5}([-a,b,-c])=\chi_{5}([a,b,c])$, we obtain $\mt_{5,7}(f; \pm 6/22, 40/44) = 380712960$. Some other examples for $\Delta=5, r=7$ are 	
 	\begin{align*}
 		\mt_{5,7}(f; \pm 13/22, 5\cdot7/44)  &= -105512960, \\
 		\mt_{5,7}(f; \pm 5/22, 5\cdot19/44) &= -17776273511920, \\
 		\mt_{5,7}(f; \pm 14/22, 5\cdot24/44) &= 789839951523840, \\
 		\mt_{5,7}(f; \pm 4/22, 5\cdot28/44)  &= 12446972332605440, \\
 		\mt_{5,7}(f; \pm 12/22, 5\cdot32/44) &= 162066199437803520, \\
 		\mt_{5,7}(f; \pm 3/22, 5\cdot35/44)  &=  -1001261756125748754,\\
 		\mt_{5,7}(f; \pm 7/22, 5\cdot39/44)  &= -10093084485445877760.
 	\end{align*}
 	It is quite amusing to explicitly construct the modular form corresponding to our theorem, similar to the forms $g_D$ given by Zagier and the Jacobi forms in \S 8 of \cite{Zagier}. The space $M_{3/2,\rho}^!$ is isomorphic to the space $J^!_{2,N}$ of weakly holomorphic Jacobi forms of weight $2$ and index $N=11$. Thus, we can as well construct it in the latter space. It is contained in the even part of the ring of weakly holomorphic Jacobi forms, which is the free polynomial algebra over $M_*^!(\SL_2(\Z))=\C[E_4,E_6,\Delta(\tau)^{-1}]/(E_4^3-E_6^2 = 1728\Delta(\tau))$ in two generators $a \in \tilde{J}_{-2,1}, b \in \tilde{J}_{0,1}$. Here $E_4$ and $E_6$ are the normalized Eisenstein series of weight 4 and 6 for $\SL_2(\Z)$. We refer to \cite{EZ}, Chapter 9 and \cite{Zagier}, \S 8 for details. The Fourier developments of $a$ and $b$ begin
	\begin{align*}
		a(\tau,z) &= (\zeta^{-1} -2 +\zeta)+(-2 \zeta^{-2} + 8 \zeta^{-1} - 12 +8 \zeta^{2}-2\zeta^{2})q + \ldots, \\
		b(\tau,z) &= (\zeta^{-1} + 10 + \zeta)+(10 \zeta^{-2}-64\zeta^{-1}+108-64\zeta+10\zeta^{2})q + \ldots,
	\end{align*}
	where $\zeta=e(z)$ and $q=e(\tau)$, as usual (we slightly abuse notation by now using $z \in \C$ as the elliptic variable for Jacobi forms).
   For $\Delta=1$ we thus obtain by Theorem \ref{thm:main} a weakly holomorphic Jacobi form $\phi^{(11)}_{1}(f;\tau,z)$ having the traces of $f$ as its Fourier coefficients. The Fourier expansion begins
   \begin{gather*}
   	-\frac{1}{2}\phi^{(11)}_{1}(f;\tau,z) = (11\zeta^{-11} + \zeta^{-1} - 24 + \zeta + 11\zeta^{11})\\
   	 + (-7256\zeta^{-6} + 885480\zeta^{-5} - 16576512\zeta^{-4} + 117966288\zeta^{-3} \\
   	 - 425691312\zeta^{-2} + 884736744\zeta^{-1} - 1122626864 + \ldots)q + \ldots.
   \end{gather*}
   The twisted traces for $\Delta=5$ as above are the coefficients of $\phi^{(11)}_{5}(f;\tau,z)$ given by
 \begin{gather*}
 	-\frac{1}{2}\phi^{(11)}_{5}(f;\tau,z) = (11\zeta^{-11} + 11\zeta^{11})q^{-11} + (\zeta^{-7} - 190356480\zeta^{-6} + 8888136755960\zeta^{-5}  \\
 	  - 6223486166302720\zeta^{-4} + 500630878062874377\zeta^{-3} - 8824913060318164992\zeta^{-2} \\
 	 + 45310559791371053140\zeta^{-1} - 77176788074781143040 + \ldots)q + \ldots.
 \end{gather*}
 It can be obtained as $\sum_{j=0}^{11} f_j a^j b^{11-j}$, where for each $j$ the function $f_j \in M^!_{2j+2}$ has a principal part starting with $a_j(-11)q^{-11}$. The Fourier expansions of these forms and their developments in terms of $E_4$, $E_6$ and $\Delta(\tau)^{-1}$, as well as some more numerical examples, can be downloaded from the second author's homepage\footnote{\url{http://www.mathematik.tu-darmstadt.de/~ehlen}}.

Finally, we consider the more general situation when $\Delta > 1$ is a fundamental discriminant, $N=p$ a prime and $f(z)=j(pz)$. 
By Theorem \ref{thm:main} together with Proposition \ref{prop:tf-} the corresponding Jacobi form $\phi^{(p)}_{\Delta}(f;\tau,z)=\sum_{n,r}c(4pn-r^2)q^n\zeta^r$
has the property that the coefficients only depend on the discriminant $4pn-r^2$ and all coefficients of negative discriminant vanish except for $c(-\Delta)=-2$ and $c(-p^2\Delta)=-2p$.

\bibliography{bib.bib}
\bibliographystyle{amsalpha}
\end{document}